\documentclass[a4paper,12pt]{article}
\usepackage[english]{babel}
\usepackage{amsmath, amsthm}
\usepackage{amssymb}

\newcommand{\1}{1\!\!\,{\rm I}}
\renewcommand{\lg}{\langle}
\newcommand{\rg}{\rangle}

\newcommand{\mbR}{{\mathbb R}}

\newcommand{\wt}{\widetilde}
\newcommand{\vf}{\varphi}
\newcommand{\ve}{\varepsilon}

\theoremstyle{plain}
\newtheorem{thm}{Theorem}
\newtheorem{lem}{Lemma}

\theoremstyle{definition}
\newtheorem{defn}{Definition}

\theoremstyle{remark}

\newcommand{\lln}{\mathop{\rm lln}}
\newcommand{\supp}{\mathop{\rm supp}}

\begin{document}\large
\renewcommand{\proofname}{Proof}
\begin{center}
{\Large\bf
Discrete time approximation of coalescing stochastic flows on the real line
}
\end{center}
\vskip20pt

I.I.Nishchenko

\vskip20pt

Abstract.   In this paper we have constructed an approximation for the Harris flow and the Arratia flow using a sequence of independent stationary Gaussian processes as a perturbation. We have established what should be the relationship between the step of approximation and smoothness of the covariance  of the perturbing processes in order to have convergence of the approximating functions to the Arratia flow.

\vskip20pt
AMS class: 60H10, 60G46
\vskip20pt

Key words:
Stochastic flow, stochastic differential equation, numerical approximation.

{\bf Introduction}. It is well-known \cite{1},  that the solution to the Cauchy
problem for SDE
\begin{equation}
\label{eq2}
\begin{cases}
dx(t)=a(x(t))dt+b(x(t))dw(t)\\
x(0)=u_0
\end{cases}
\end{equation}
with continuously differentiable
functions $a, b$   having bounded
derivatives, can be obtained via discrete time
approximation. Namely, if we define the sequence $\{x^m_n\}$   by the rule:
\begin{equation}
\label{eq2new}
x^m_0=x_0\in\mbR, \
x^m_{n+1}=x^m_n+\frac{1}{m}a(x^m_n)+\frac{1}{\sqrt{m}}b(x^m_n)\xi_n
\end{equation}
where $\{\xi_n, n\geq1\}$ is a sequence of independent standard
Gaussian random variables, then the random functions
$$
x_m(t)=
m\bigg(\frac{k+1}{m}-t\bigg)x^m_k +
m\bigg(t-\frac{k}{m}\bigg)x^m_{k +1}, \
t\in\bigg[
\frac{k}{m}; \frac{k+1}{m}
\bigg], k=0, \ldots, m-1
$$
weakly converge in $C([0, 1])$ to the solution of \eqref{eq2}.

In this paper we study similar to \eqref{eq2new} difference approximation
for coalescing stochastic flows.   As is
known \cite{2}, such flows are not generated by a Gaussian white noise in
the space of vector fields. In order to understand how the flow with
coalescence is arranged we can consider its difference approximation. As a
perturbation we select a sequence of Gaussian stationary processes. In
order to allow the coalescence of the trajectories of individual particles
in the limit, the covariance functions of these processes are chosen to be
less and less smooth at the origin. On the other hand, in order the limit
flow to preserve the order, the step of approximation must be sufficiently
small. The relationship between the step of approximation and smoothness of
the covariance of the perturbing processes explains to some extent the
structure of singular stochastic flows.

{\bf 1. SDE and stochastic flows on the real line}. The main object of the
article is the Harris flow of Brownian motions on $\mbR.$  Let $\vf$ be a
continuous real positive definite function on $\mbR$ such that $\vf(0)=1$ and $\vf$  is Lipschitz outside any neighborhood of zero.

\begin{defn}
\label{defn1}
The Harris flow with $\vf$ being its local characteristic  is a family
$\{x(u, \cdot); u\in\mbR\}$  of Brownian martingales with respect to the
joint filtration such, that

1) for every $u_1\leq u_2$  and $t\geq0$
$$
x(u_1, t)\leq x(u_2, t),
$$

2) the joint characteristics are:
$$
d\lg x(u_1, \cdot), x(u_2, \cdot)\rg(t)=\vf(x(u_1, t)-x(u_2, t))dt.
$$
\end{defn}

It is known that the Harris flow exists \cite{3}. If the function $\vf$ is
smooth enough, the Harris flow  can be obtained as a flow of solutions to SDE.
Namely, for a sequence of standard Wiener processes $\{w_k; k\geq1\}$
consider the following SDE
\begin{equation}
\label{eq3}
dx(u, t)=\sum^\infty_{k=1}a_k(x(u, t)) dw_k(t),
\end{equation}
where $a=(a_k)_{k\geq1}$ is a Lipschitz mapping from $\mbR$ to $l_2$  such
that
$$
\sum^\infty_{k=1}a^2_k\equiv1,
$$
and
$$
\sum^\infty_{k=1}a_k(u)a_k(v)=\vf(u-v).
$$

Then the flow corresponding to \eqref{eq3} is the Harris flow with the local characteristic $\vf$, and furthermore it is a flow of homeomorphisms.
  Note, that the Harris flow could be
coalescent \cite{3}  and, in this case may not be generated by SDE. By this
reason it is interesting to consider  discrete approximations for the
flow  built in a similar way as approximations to SDE. Consider
a sequence of independent stationary Gaussian processes $\{\xi_n(u);  u\in\mbR, n\geq1\}$
with zero mean and covariation function $\Gamma.$  Suppose, that
$\Gamma $ is continuous. Define a sequence of random mappings
$\{x_n; n\geq0\}$  by the rule
\begin{equation}
\label{eq4}
x_0(u)=u, \ x_{n+1}(u)=x_n(u)+\xi_{n+1}(x_n(u)), \ u\in\mbR.
\end{equation}
Note, that  the continuity of $\Gamma$ implies that the processes $\{\xi_n; n\geq1\}$
have measurable modifications. This allows to substitute $x_n$  into
$\xi_{n+1}.$  The independence of $\{\xi_n; n\geq1\}$  guarantees that
$\xi_{n+1}(x_n(u))$  does not depend on the choice of these modifications.
We will need the following description of one and two-point motions of $\{x_n; n\geq0\}$.

\begin{lem}
\label{lem1}
The sequences
$\{x_n(u); n\geq0\}$   and
$\{x_n(u_2)-x_n(u_1); n\geq0\}$
have the same distributions as the sequences $\{y_n(u); n\geq0\},$ $\{z_n(u); n\geq0\}$, which  are defined by the
following rules:
$$
y_0=u, \ y_{n+1}=y_n+\eta_n,
$$
$$
z_0=u_2-u_1, \ z_{n+1}=z_n+\sqrt{2\Gamma(0)-2\Gamma(z_n)}\eta_n,
$$
where $\{\eta_n; n\geq1\}$   is a sequence of independent standard normal variables.
\end{lem}

The proof of the lemma can be obtained easily by calculating
conditional distributions of $x_{n+1}$  under given $x_0, \ldots, x_n,$
and is omitted.

It follows from Lemma \ref{lem1}, that the sequence of random
mappings $\{x_n; n\geq0\}$  is similar to the  Harris flow. All its one-point
motions are Gaussian symmetric random walks. But  the  mappings
$x_n$  for $n\geq1$   are not monotone. In the next section we will prove
that any $m$-point motion of $\{x_n; n\geq0\}$   approximates  
the $m$-point motion of the Harris flow.

{\bf 2. $m$-point motions}. In this section we will consider the limit behavior of
$x_n$  under a suitable normalization. Let us define the random
functions
$$
\wt{x}_n(u, t)=
n\left(\frac{k+1}{n}-t\right)x_k(u)+
n\left(t-\frac{k}{n}\right)x_{k+1}(u), \
$$
$$
 u\in \mbR, \
t\in\left[\frac{k}{n}; \frac{k+1}{n}\right], k=0, \ldots, n-1.
$$
Our first result is related to the $n$-point motions of $\wt{x}_n.$

\begin{thm}
\label{thm1}
Let $\Gamma$ be continuous positive definite function on $\mbR$ such that $\Gamma(0)=1$ and $\Gamma$ has two continuous bounded derivatives. Suppose that $\wt{x}_n$ is built upon a sequence $\{\xi_k; k\geq1\}$  with covariance $\frac{1}{\sqrt{n}}\Gamma.$

Then for every $u_1,\ldots, u_l\in\mbR$ the random processes
$\{\wt{x}_n(u_j, \cdot), j=1, \ldots, l\}$
weakly converge in $C([0; 1], \mbR^l)$
to the $l$-point motion of the Harris flow with the local characteristic $\Gamma$.
\end{thm}

\begin{proof}
It follows from Lemma \ref{lem1} and the invariance principle, that for every
$j=1, \ldots, l$  \ $\wt{x}_n(u_j, \cdot)$   weakly converges in
$C([0; 1])$   to the Brownian motion which starts from $u_j.$   Then, it
remains to prove that any limit point of $\{\wt{x}_n(u_j, \cdot), j=1,
\ldots, l\}$  coincides with the $l$-point motion of the Harris flow. Without
loss of generality suppose that the whole sequence $\{\wt{x}_n(u_j, \cdot),
j=1, \ldots, l\}$   weakly converges. For a function
$f\in C^3(\mbR)$  with bounded derivatives, consider the random processes
$$
y_n(t)=\wt{x}_n(u_{j+1}, t)-\wt{x}_n(u_j, t),
$$
$$
z_n(t)=f(y_n(t))-f(u_{j+1}-u_j)-\int^t_0(1-\Gamma(y_n(s)))f''(s)ds.
$$
Following the known procedure (see for example \cite{4}), it is easy to verify,
that  $\{z_n; n\geq1\}$  weakly converges to a certain martingale.
Consequently the weak limit of $y_n$  satisfies the martingale problem
for the operator
$$
Af(x)=(1-\Gamma(x))\frac{d^2}{dx^2}f(x).
$$
Since the martingale problem now has a unique solution \cite{4}, then the
weak limit of $y_n$  is the solution to the following Cauchy problem
$$
\begin{cases}
dy(t)=\sqrt{2-2\Gamma(y(t))}dw(t),\\
y(0)=u_{j+1}-u_j.
\end{cases}
$$
The solution to this SDE has the strong Markov property. Consequently $y$  is
nonnegative for $u_{j+1}-u_j>0.$   Hence, the weak limit of $\{\wt{x}_n(u_j,
\cdot); j=1, \ldots, l\}$  preserves the order. It remains to check the form
of the joint cha\-rac\-te\-ris\-tic, which can be done in a standard way. The theorem is
proved.
\end{proof}

The previous result is based on the uniqueness of a solution to SDE related to
a stochastic flow. Now we consider the convergence of difference
approximations
to the $n$-point motions of the Arratia flow. Let us recall that Arratia's flow
\cite{5} is the Harris flow with the local characteristic $\Gamma=\1_{\{0\}}.$   
In this flow any two trajectories coalesce into a single one in finite time.
\begin{thm}
\label{thm2}
Suppose, that for every $m\geq1 \ \wt{x}_m$  is built upon a
sequence $\{\xi^m_n; n\geq1\}$   where  independent identically
distributed processes $\xi^m_n$  have the covariance function $\Gamma_m$
which satisfies the Lipschitz condition. Define for $m\geq1$
$$
C_m=\sup_{\mbR}\frac{2-2\Gamma_m(x)}{x^2}.
$$
If

1) $\lim_{m\to\infty}\frac{C_me^{C_m}}{m}=0,$

2) for every $\delta>0$ \ $\sup_{\mbR\setminus[-\delta; \delta]}|\Gamma_m(x)|
\to0, m\to\infty,$

  then the random processes
$\{\wt{x}_m(u_1, \cdot), \ldots, \wt{x}_m(u_l, \cdot); m\geq1\}$ weakly converge to the $l$-point motion of
Arratia's flow starting from $u_1, \ldots, u_l.$
\end{thm}

\begin{proof}
As in the proof of Theorem \ref{thm1}  we have the weak compactness of\newline
$\{(\wt{x}_m(u_1, \cdot), \ldots, \wt{x}_m(u_l, \cdot); m\geq1)\}$
in
$C([0; 1], \mbR^l)$
and the weak convergence of $x_m(u_i, \cdot)$   to a Wiener process.
Consequently, for any limit point of \\
$\{(\wt{x}_m(u_1, \cdot), \ldots, \wt{x}_m(u_l, \cdot); m\geq1)\}$
it is enough to check the mutual characteristics and the order
preserving property. For $u_i<u_{i+1}$   the difference process
$y_m(t)=\wt{x}_m(u_{i+1}, t)-\wt{x}_m(u_{i+1}, t)$
are equidistributed with the difference approximation $v_m$ to the solution of the SDE
$$
\begin{cases}
d\wt{y}_m(t)=\sqrt{2-2\Gamma_m(\wt{y}_m(t))}dw(t),\\
\wt{y}_m(0)=u_{i+1}-u_i.
\end{cases}
$$
It is known \cite{1}, that
$$
E\sup_{[0; 1]}(v_m(t)-\wt{y}_m(t))^2\leq
C\frac{C_me^{C_m}}{m}.
$$
Note, that $\wt{y}_m$  is nonnegative. Consequently, for every $r>0$
$$
P\{\inf_{[0; 1]}y_m<-r\}=P\{\inf_{[0; 1]}v_m<-r\}\to0, m\to\infty.
$$
Hence the weak limit of any subsequence of $\{y_m; m\geq1\}$  is nonnegative.
The completion of the proof can be done exactly as in the previous theorem using
martingale approximation and the fact that any nonnegative martingale remains at zero 
after hitting zero. The theorem is proved.
\end{proof}

{\bf 3. Convergence of random  maps}

In this section we will consider convergence of $\{\wt{x}_n; n\geq1\}$   as  random maps
to  corresponding maps from a  stochastic flow. Let us begin with the
 case of smooth $\Gamma.$ Define  the sequence
\begin{equation}
\label{eq5}
x^m_{n+1}(u)=x^m_n(u)+\frac{1}{\sqrt{m}}\xi_{n+1}(x^m_n(u)),
\end{equation}
where $\{\xi_n; n\geq1\}$ is a sequence of independent  stationary centered Gaussian
processes with covariance function $\Gamma$
satisfying the inequality
$$
\forall \ u\in\mbR: \ 1-\Gamma(u)\leq Cu^2
$$
with some constant $C.$ Define the Harris flow $x$ corresponding to $\Gamma.$
Note, that now $x$ has a modification $x(u, t), u\in\mbR, t\in[0; 1]$
continuous with respect to both variables. Really, using the martingale
inequality one can get that
$$
E\sup_{s\in[0; t]}(x(u, s)-x(v,s))^2\leq
$$
$$
\leq 2(u-v)^2+2E\int^t_0(2-2\Gamma(x(u, s)-x(v, s)))ds\leq
$$
$$
\leq 2(u-v)^2+4C\int^t_0\sup_{r\in[0; s]}(x(u, r)-x(v, r))^2ds.
$$
Consequently, for some $\wt{c}$
$$
E\sup_{t\in[0; 1]}(x(u, t)-x(v, t))^2\leq\wt{c}(u-v)^2.
$$
This inequality together with the Kolmogorov condition gives us the desired property.

The next statement asserts the convergence of our approximations to a stochastic
flow in the case of smooth $\Gamma.$
\begin{thm}
\label{thm3}
The random functions
$\{\wt{x}_m=x^m_m; m\geq1\}$
converge in distribution in the space $C([a; b])$ to the random function
$x$  for arbitrary interval $[a; b].$
\end{thm}

\begin{proof}
The convergence of finite-dimensional distributions was proved in
Theorem \ref{thm1}. It remains to check the weak compactness of
$\{\wt{x}_m; m\geq1\}.$   For arbitrary $u, v\in\mbR$    we have
$$
E(x^m_{n+1}(u)-x^m_{n+1}(v))^2=
E(x^m_{n}(u)-x^m_{n}(v))^2+
\frac{1}{m}E(2-2\Gamma(x^m_{n}(u)-x^m_{n}(v)))\leq
$$
$$
\leq
E(x^m_{n}(u)-x^m_{n}(v))^2+ 2Cm^{-1}E(x^m_{n}(u)-x^m_{n}(v))^2.
$$
Consequently,
$$
E(\wt{x}_{m}(u)-\wt{x}_{m}(v))^2\leq(u-v)^2
(1+\frac{2C}{m})^m\leq e^{2C}(u-v)^2.
$$
The obtained estimation gives the desired weak compactness. The theorem is
proved.
\end{proof}

To obtain approximation of Arratia's flow we need some
additional results about the convergence of smooth stochastic flows to Arratia's flow.
Let us consider the following SDE with the space-time white noise (Wiener
sheet) $W$
\begin{equation}
\label{eq6}
\begin{split}
&dz(u, t)=\int_{\mbR}\vf(z(u, t)-p)W(dp, dt),\\
&
z(u,0)=u, \ u\in\mbR,
\end{split}
\end{equation}
where $\vf\in C^\infty_0(\mbR)$   and $\int_{\mbR}\vf^2(u)du=1$ (see
\cite{6, 7}  about equations of type \eqref{eq6}). All what we need here is
a statement, that under our condition on $\vf$  the unique strong
solution to \eqref{eq6}  exists and is the Harris flow corresponding to the
local characteristic
$$
\Gamma(u)=\int_{\mbR}\vf(-p)\vf(u-p)dp.
$$
It was proved in \cite{8}, that the $n$-point motions of solutions $z_\ve$  to
\eqref{eq6}  which corresponds to $\vf_\ve$    with the property
$\supp\vf_\ve\subset[-\ve; \ve]$  converge in distribution to the $n$-point
motions of the Arratia flow when $\ve\to0.$

Consider discrete approximations of $z.$   For every $n\geq1$  define
\begin{equation}
\label{eq8}
\begin{split}
&z^n_0(u)=u,\\
&
z^n_{k+1}(u)=z^n_k(u)+
\int^{\frac{k+1}{n}}_{\frac{k}{n}}\int_{\mbR}
\vf(z^n_k(u)-p)W(dp, dt), \\
&
k=0, \ldots, n-1.
\end{split}
\end{equation}
It can be easily checked that every $z^n_k$  has a continuous modification.
The next theorem gives a speed of convergence of $z^n_n$  to $z(\cdot, 1)$
in the space $C([0; 1]).$   Define
$$
L^2=\int_{\mbR}\vf'(p)^2dp.
$$

\begin{thm}
\label{thm3'}
There exist such positive constants  $C', C'', C''',$  that for every $n\geq1$
\begin{equation}
\label{eq9}
E\|z^n_n-z(\cdot, 1)\|\leq \frac{C'}{\sqrt{n}}
\exp\{(C''L^2+C'''L^4)e^{4L^2}+L^2\}(L^2+1).
\end{equation}
where $\|\cdot\|$  is the uniform norm in $C([0; 1]).$
\end{thm}

\begin{proof}
Consider for $k=1, \ldots, n$
$$
E\left(z^n_k(0)-z\left(0, \frac{k}{n}\right)\right)^2=
E\left(z^n_{k-1}(0)-z\left(0, \frac{k-1}{n}\right)\right)^2+
$$
$$
+E
\int^{\frac{k}{n}}_{\frac{k-1}{n}}\int_{\mbR}
(\vf(z^n_{k-1}(0)-p)-\vf(z(0, s)-p))^2dpds\leq
$$
$$
\leq E
\left(z^n_{k-1}(0)-z\left(0, \frac{k-1}{n}\right)\right)^2+
L^2E
\int^{\frac{k}{n}}_{\frac{k-1}{n}}
(z^n_{k-1}(0)-z(0, s))^2ds=
$$
$$
=
E\left(z^n_{k-1}(0)-z\left(0, \frac{k-1}{n}\right)\right)^2
\left(1+L^2\frac{1}{n}\right)+
$$
$$
+
L^2E
\int^{\frac{k}{n}}_{\frac{k-1}{n}}
\left(z(0)-z\left(0, \frac{k-1}{n}\right)\right)^2ds=
$$
$$
=
E\left(z^n_{k-1}(0)-z\left(0, \frac{k-1}{n}\right)\right)^2
\left(1+\frac{L^2}{n}\right)+\frac{L^2}{2n^2}.
$$
Consequently,
$$
E(z^n_n(0)-z(0, 1))^2\leq \frac{L^2}{n^2}e^{L^2}.
$$
Note, that under our conditions on $\vf,$  random functions $\{z^n_k\}$  and
$z$   have  continuous derivatives with respect to the spatial variable.

Let us denote by $y^n_k$  and $y$  these derivatives. Then for
$k=1, \ldots, n$
$$
y^n_k(u)=y^n_{k-1}(u)
\left(
1+
\int^{\frac{k}{n}}_{\frac{k-1}{n}}\int_{\mbR}
\vf'(z^n_{k-1}(u)-p)W(dp,dt)\right),
$$
and
$$
dy(u, t)=y(u, t)\int_{\mbR}\vf'(z(u, t)-p)W(dp, dt).
$$
Hence,
$$
y^n_k(u)-
y\left(u, \frac{k}{n}\right)=
y^n_{k-1}(u)-
y\left(u, \frac{k}{n}\right)+
$$
$$
+
\int^{\frac{k}{n}}_{\frac{k-1}{n}}\int_{\mbR}
[y^n_{k-1}(u)\vf'(z^n_{k-1}(u)-p)-y(u, s)\vf'(z(u, s)-p)]W(dp, dt)=
$$
$$
=
y^n_{k-1}(u)-
y\left(u, \frac{k-1}{n}\right)+
$$
$$
+
\int^{\frac{k}{n}}_{\frac{k-1}{n}}\int_{\mbR}
\left[\left(y^n_{k-1}(u)-
y\left(u, \frac{k-1}{n}\right)
\right)
\vf'(z^n_{k-1}(u)-p)+
\right.
$$
$$
+
y\left(u, \frac{k-1}{n}\right)
\left(\vf'(z^n_{k-1}(u)-p\right)-
\vf'\left(z\left(u, \frac{k-1}{n}-p\right)+
\right.
$$
$$
+
y\left(u, \frac{k-1}{n}\right)
\left(\vf'\left(z\left(u, \frac{k-1}{n}\right)-p\right)-
\vf'(z(u, s)-p)\right)+
$$
$$
\left.
+\vf'(z(u, s)-p)
\left(y\left(u,\frac{k-1}{n}\right)-y(u, s)\right)\right] W(dp, ds).
$$
Then
$$
E
\left(
y^n_k(u)-y\left(u, \frac{k}{n}\right)\right)^2=
E
\left(
y^n_{k-1}(u)-y\left(u, \frac{k-1}{n}\right)\right)^2+
$$
$$
+
4E
\int^{\frac{k}{n}}_{\frac{k-1}{n}}\int_{\mbR}
\left(
y^n_{k-1}(u)-y\left(u, \frac{k-1}{n}\right)\right)^2
\vf'(z^n_{k-1}(u)-p)^2dpds+
$$
$$
+
4E
\int^{\frac{k}{n}}_{\frac{k-1}{n}}\int_{\mbR}
y\left(u, \frac{k-1}{n}\right)^2
(\vf'(z^n_{k-1}(u)-p)-
\vf'\left(z\left(u, \frac{k-1}{n}\right)-p\right)^2dpds+
$$
$$
+
4E
\int^{\frac{k}{n}}_{\frac{k-1}{n}}\int_{\mbR}
y\left(u, \frac{k-1}{n}\right)^2
\left(\vf'(z\left(u,\frac{k-1}{n}\right)-p\right)-
\vf'(z(u, s)-p))^2dpds+
$$
$$
+
4E
\int^{\frac{k}{n}}_{\frac{k-1}{n}}
\int_{\mbR}
\vf'(z(u, s)-p)^2
\left(y\left(u, \frac{k-1}{n}\right)-y(u, s)\right)^2dpds\leq
$$
$$
\leq
E
\left(y^n_{k-1}(u)-y\left(u, \frac{k-1}{n}\right)\right)^2\cdot
(1+\frac{4}{n}L^2)+
$$
$$
+\frac{4}{n}E
y\left(u, \frac{k-1}{n}\right)^2\cdot L^2
\left(
z^n_{k-1}(u)-z\left(u, \frac{k-1}{n}\right)
\right)^2+
$$
$$
+
4E
 y\left(u, \frac{k-1}{n}\right)^2L^2
\int^{\frac{k}{n}}_{\frac{k-1}{n}}
\left(
z\left(u, \frac{k-1}{n}\right)
-z(u, s)
\right)^2ds+
$$
$$
+
4L^2E
\int^{\frac{k}{n}}_{\frac{k-1}{n}}
 \left
(y\left(u, \frac{k-1}{n}\right)-y(u, s)\right)^2ds.
$$
Note, that the processes $z(u, t), t\in[0; 1]$  and
$$
\eta(t)=\int^t_0\int_{\mbR}\vf'(z(u, s)-p)W(dp, ds), \ t\in[0; 1]
$$
are continuous martingales with the characteristics
$$
\lg z(u, \cdot)\rg(t)=t, \
\lg\eta\rg(t)=L^2t.
$$
Consequently, $z(u, \cdot)$  and $\eta$  are Wiener processes. It
follows from this that
$$
y(u, t)=\exp\{\eta(t)-\frac{t}{2}L^2\}.
$$
Hence
$$
E
y\left(u, \frac{k-1}{n}\right)^2
\int^{\frac{k}{n}}_{\frac{k-1}{n}}
 \left
(z\left(u, \frac{k-1}{n}\right)-z(u, s)\right)^2ds=
$$
$$
=
E
y\left(u, \frac{k-1}{n}\right)^2
E
\int^{\frac{k}{n}}_{\frac{k-1}{n}}
 \left
(z\left(u, \frac{k-1}{n}\right)-z(u, s)\right)^2ds\leq
$$
$$
\leq
\frac{1}{2n^2}e^{L^2},
$$
$$
E
\int^{\frac{k}{n}}_{\frac{k-1}{n}}
 \left
(y\left(u, \frac{k-1}{n}\right)-y(u, s)\right)^2ds=
$$
$$
=
E
\int^{\frac{k}{n}}_{\frac{k-1}{n}}
\left(
\int^s_{\frac{k-1}{n}}\int_{\mbR}
y(u, r)\vf'(z(u, r)-p)W(dp, dr)\right)^2ds=
$$
$$
=L^2 E
\int^{\frac{k}{n}}_{\frac{k-1}{n}}
\int^s_{\frac{k-1}{n}}
y(u, r)^2drds\leq
$$
$$
\leq
\frac{1}{2n^2}L^2\cdot e^{L^2}.
$$
Furthermore,
$$
E
y\left(u, \frac{k-1}{n}\right)^2
\left(
z^n_{k-1}(u)-z\left(u, \frac{k-1}{n}\right)
\right)^2\leq
$$
$$
\leq
\sqrt{E y\left(u, \frac{k-1}{n}\right)^4}
\sqrt{E\left(z^n_{k-1}(u)-z\left(u, \frac{k-1}{n}\right)\right)^4}\leq
$$
$$
\leq
e^{3L^2}C_2\frac{L^2}{n}e^{L^2}.
$$

In the last inequality the martingale property of $x^n$ and $x$  was used.

Finally one can get
$$
E\left(
y^n_k(u)-y\left(u, \frac{k}{n}\right)
\right)^2\leq
$$
$$
\leq
E\left(
y^n_{k-1}(u)-y\left(u, \frac{k}{n}\right)
\right)^2
\left(1+\frac{4L^2}{n}+\frac{4}{n}L^4C_2e^{4L^2}\right)+
\frac{c_3}{n^2}(L^2+1)e^{L^2}.
$$
Consequently,
$$
E(y^n_n(u)-y(u, 1))^2\leq
$$
$$
\leq
\frac{c_4}{n}\exp\{(c_5L^2+c_6L^4)e^{4L^2}+L^2\}(L^2+1).
$$
To obtain an estimation for the uniform norm  $\|z^n_n-z(\cdot, 1)\|$ we
proceed as follows
$$
E\|z^n_n-z(\cdot, 1)\|\leq E
|z^n_n(0)-z(0, 1)|+E\int^1_0|y^n_n(u)-y(u, 1)|du\leq
$$
$$
\leq
\frac{c_7}{\sqrt{n}}\exp\{(c_8L^2+c_9L^4)e^{4L^2}+L^2\}(L^2+1).
$$
The theorem is proved.
\end{proof}

Obtained estimation can be used to get the convergence of the difference
approximation to Arratia's flow. We will establish this convergence using the L\'evy--Prokhorov distance. Let us recall its definition.

\begin{defn}
\label{defn1} \cite{9}.
For two nondecreasing c\`adl\`ag functions $f, g$  on $[0; 1]$ the
L\'evy--Prokhorov distance is
$$
\begin{aligned}
\rho(f, g)=&\inf\{\ve>0: \ \forall \ u\in[0; 1]: \\
&
f(u-\ve)-\ve\leq g(u)\leq f(u+\ve)+\ve\\
&
g(u-\ve)-\ve\leq f(u)\leq g(u+\ve)+\ve\}.
\end{aligned}
$$
\end{defn}

It is well-known \cite{9}  that the convergence in this distance is
equivalent to the convergence at every point of continuity of the limit
function. Also note that
$$
\rho(f, g)\geq d(f, g),
$$
where $d(f, g)$ is the Skorokhod distance between $f$  and $g$  \cite{9}.

Take a function $\psi\in C^\infty_0$   with $\supp\psi\subset[-1; 1]$
such, that
$$
\int_{\mbR}\psi^2(u)du=1.
$$
For arbitrary $\ve>0$   define
$$
\psi_\ve(u)=\frac{1}{\ve^{1/2}}\psi\left(\frac{u}{\ve}\right),
$$
$$
\Gamma_\ve(u)=\frac{1}{\ve}\int_{\mbR}\psi_\ve(p)\psi_\ve(u+p)dp.
$$
The parameter $\ve$ here is associated with the smoothness of $\Gamma_\ve$. 
In order to approximate the Arratia flow we have to take $\ve\to 0$.
For independent Gaussian processes $\{\xi_n; n\geq1\}$  with the covariance
$\{\Gamma_{\ve_n}\}$    let us construct the sequences
$$
x^n_{k+1}(u)=x^n_k(u)+\frac{1}{\sqrt{n}}\xi_n(x^n_k(u)).
$$
The next theorem shows that $x^n_n$   can be used to approximate the
Arratia flow.

\begin{thm}
\label{thm4}
Suppose that $\ve_n\to0, n\to\infty,$
$$
\frac{1}{\ve^2_n}=o(\lln n), \ n\to\infty.
$$
Then the random functions $x^n_n$   converge weakly in $D([0; 1])$   to the
value of the Arratia flow $x(\cdot, 1).$
\end{thm}
\begin{proof}
Consider the sequence of SDE
$$
dz_{\ve_n}(u, t)=\int_{\mbR}\psi_\ve(z_{\ve_n}(u, t)-p)W(dp, dt).
$$
As it was mentioned in the beginning of this section, for every
$u_1, \ldots, u_m\in[0; 1]$
$(z_{\ve_n}(u_1, 1), \ldots, z_{\ve_n}(u_m, 1))$
weakly converge to
$(x(u_1, 1), \ldots, x(u_m, 1)).$   Hence \cite{10},
$z_{\ve_n}(\cdot, 1)$  weakly converge to $x(\cdot, 1)$  in the
L\'evy--Prokhorov distance. For every $n\geq1$  the sequence
$x^n_1, \ldots, x^n_n$   is equidistributed with the discrete approximations
to $z_{\ve_n}$ from Theorem \ref{thm3'}. Consequently, $x^n_n$   is
equidistributed with
$\wt{x}_n$   such that
$$
E\|\wt{x}_n-z_{\ve_n}(\cdot, 1)\|\leq \frac{C'}{\sqrt{n}}
\exp\{(C''L^2_{\ve_n}+C'''L^4_{\ve_n})e^{4L^2_{\ve_n}}+L^2_{\ve_n}\}(L^2_{\ve_n}+1),
$$
where $L^2_{\ve_n}=\frac{1}{\ve_n}\int_{\mbR}\,\psi'(p)^2dp$. Hence 
$$
E\|\wt{x}_n-z_{\ve_n}(\cdot, 1)\|\to 0,\,n\to\infty.
$$
Since for continuous functions $f, g$ the Skorokhod distance
$$
d(f, g)\leq \|f-g\|,
$$
then $x^n_n$   weakly converges to $x(\cdot, 1)$  in $D([0; 1]).$
The theorem is proved.
\end{proof}

\end{document}